\documentclass[11pt]{amsart}
\usepackage{amsfonts,amsmath,amsthm,amscd,amssymb,latexsym,amsbsy,pb-diagram,cite}
\usepackage{graphicx}
\usepackage{tikz}
\usepackage{url}
 \usepackage{tkz-graph}
 \usetikzlibrary{3d,calc,arrows,decorations,positioning,shapes}
\newcounter{minutes}
\divide\time by 60
\newcounter{hours}
\multiply\time by 60 \addtocounter{minutes}{-\time}

\newtheorem{theorem}[equation]{Theorem}

\newtheorem{proposition}[equation]{Proposition}

\newtheorem{corollary}[equation]{Corollary}

\theoremstyle{definition}%
\newtheorem{definition}[equation]{Definition}%
\newtheorem{example}[equation]{Example}
\numberwithin{equation}{section}

\DeclareMathOperator{\ve}{\varepsilon}
 \DeclareMathOperator{\card}{card}

\newcommand{\NN}{\mathbb N}

\begin{document}

\title[On the uniqueness of continuation of a metric]{\bf On the uniqueness of continuation of a partially defined metric
}

\author{\bf Evgeniy Petrov}

\address{Institute of Applied Mathematics and Mechanics of the NAS of Ukraine, Dobrovolskogo str. 1, Slovyansk 84100, Ukraine}
\email{eugeniy.petrov@gmail.com}

\begin{abstract}
The problem of continuation of a partially defined metric can be efficiently studied using graph theory.
Let $G=G(V,E)$ be an undirected graph with the set of vertices $V$ and the set of edges $E$. \mbox{A necessary} and sufficient condition under which the weight $w\colon E\to\mathbb R^+$ on the graph $G$ has a unique continuation to a metric $d\colon V\times V\to\mathbb R^+$ is found.
 \\\\
{\bf Key words:} weighted graph, metric space,
shortest-path metric, continuations of metrics.\\\\
{\bf 2010 Mathematics Subject Classification:} 05C12, 54E35
\end{abstract}
\maketitle

\section{Introduction}
\footnote{This is the submitted version. The published version is available at Theory and Applications of Graphs: Vol. 10: Iss. 1, Article 1, 2023. \\ \url{https://doi.org/10.20429/tag.2022.100101}}
The problem of continuation of a weight $w\colon E\to\mathbb R^+$ defined on the set of edges $E(G)$ of the graph $G=(V,E)$ to a pseudometric was considered in ~\cite{DMV}. In particular, it was found a set of necessary and sufficient conditions under which the weight $w$ can be extended to a pseudometric $d\colon V\times V\to\mathbb R^+$, see Theorem~\ref{t1}. The set $\mathfrak M_w$ of all such extensions can be partially ordered, see Definition~\ref{d1}.
It was shown that the shortest-path pseudometric $d_w$ is the greatest element of $\mathfrak M_w$ and that the least continuation exist if and only if $G$ is complete $k$-partite. The analogous problem of continuation of a weight to an ultrametric was considered in~\cite{DP}. Moreover, in~\cite{DP} the question of uniqueness of such continuation was studied. The aim of this paper is to fill this gap in~\cite{DMV}, i.e., to find a uniqueness criterion for continuation of a weight to a metric (pseudometric).

Note that a problem of continuation of a weight to a metric can be reformulated as a problem of continuation of a partially defined metric. In such setting, some cases of this problem were considered earlier. For example, the free amalgamation property for finite metric spaces states that there always exists a metric on the union $X\cup Y$ of finite metric spaces $(X,d_1)$, $(Y,d_2)$ which agrees with $d_1$ on $X$ and $d_2$ on $Y$, if $d_1=d_2$ for elements of $Z$ where $Z=X\cap Y$, see~\cite{B,AMI}. A review of results related to extensions of continuous and uniformly continuous pseudometrics can be found in~\cite{H}.

Recall the basic definitions from the graph theory, see, for example, \cite{BM}. A \textit{graph} is a pair $(V,E)$ consisting of a nonempty set $V$ and a (probably empty) set $E$  elements of which are unordered pairs of different points from $V$. For a graph $G=(V,E)$, the sets $V=V(G)$ and $E=E(G)$ are called \textit{the set of vertices} and \textit{the set of edges}, respectively.  A \emph{path} in a graph $G$ is a subgraph $P$ of $G$ whose vertices can be numbered so that
$$
V(P)=\{x_0,x_1,...,x_k\}, \quad E(P) =\{\{x_0,x_1\},...,\{x_{k-1},x_k\}\},
$$
where all $x_i$ are distinct. A graph $G$ is \emph{connected} if any two distinct vertices of $G$ can be joined by a path. A finite graph $C$ is a \emph{cycle} if $|V(C)| \geq 3$ and there exists an enumeration $(v_1, \ldots, v_n)$ of its vertices such that
$$
(\{v_i, v_j\} \in E(C)) \Leftrightarrow (|i-j|=1 \text{ or } |i-j|=n-1).
$$
The edge $e=\{u,v\}$ is said to {\it join} $u$ and $v$, and the vertices $u$ and $v$ are called {\it adjacent} in $G$. The graph $G$ is {\it empty} if no two vertices are adjacent, i.e., if $E(G)=\varnothing$.

A {\it weighted graph} $(G,w)$ is a graph $G=(V,E)$ together with a weight $w\colon E\to\mathbb R^+$ where $\mathbb{R}^+=[0,\infty)$. If $(G,w)$ is a weighted graph, then for each subgraph $F$ of the graph $G$  we define the weight of $F$ as
\begin{equation}\label{e1}
w(F)=\sum_{e\in E(F)}w(e).
\end{equation}

Recall also that a {\it pseudometric} $d$ on the set $X$ is a
function $d:X\times X\to\mathbb R^+$ such that $d(x,x)=0,\
d(x,y)=d(y,x)$ and $d(x,y)\leq d(x,z)+d(z,y)$ for all $x,y,z\in X$.
The pseudometric $d$ on $X$ is a {\it metric} if, in addition,
$((d(x,y)=0)\Rightarrow(x=y))$
for all $x,y\in X$.

Let $(G,w)$ be a weighted graph.
If there exists a pseudometric $d$ on the set $V(G)$ such that the equality
\begin{equation}\label{e28}
w(\{u,v\})=d(u,v)
\end{equation}
holds for all edges $\{u,v\}\in E(G)$, then we say that the weight $w$ is {\it pseudometrizable}.

Let $(G,w)$ be a connected  weighted graph and let $u,v \in V(G)$. Denote by $\mathcal
P_{u,v}=\mathcal P_{u,v}(G)$ the set of all paths joining  $u$ and
$v$ in $G$. Write
\begin{equation}\label{e3}
d_w(u,v)=\inf\{w(P):P\in\mathcal P_{u,v}\},
\end{equation}
where $w(P)$ is the weight of the path $P$. The function $d_w$ is known as {\it weighted shortest-path pseudometric}.

We need the following for the proof of the main result of this paper.

\begin{theorem}[\!\!\cite{DMV}]\label{t1}
Let $(G,w)$ be a weighted graph. The following statements are
equivalent.
\begin{itemize}

\item[\rm(i)] The weight $w$ is pseudometrizable.

\item[\rm(ii)] The equality
\begin{equation}\label{e4}
w(\{u,v\})=d_w(u,v)
\end{equation}
holds for all $\{u,v\}\in E(G)$.

\item[\rm(iii)] For every cycle $C\subseteq G$ we have the
inequality
\begin{equation}\label{eq5}
2\max_{e\in E(C)}w(e)\leq w(C).
\end{equation}
\end{itemize}
\end{theorem}

\begin{definition}\label{d1}
Let $G$ be a graph and let $w$ be a pseudometrizable weight on $E(G)$.
Denote by $\mathfrak M_w$ is the set of all pseudometrics $\rho$
on $V(G)$ satisfying the equality
\begin{equation}\label{e14}
\rho(u,v)=w(\{u,v\})
\end{equation}
for each $\{u,v\}\in E(G)$. Let us introduce a partial order
$\leqslant$ on the set $\mathfrak M_w$ as:
\begin{equation*}\label{e14*}
(\rho_1\leqslant\rho_2)\quad \text{if and only if}\quad
(\rho_1(u,v)\leqslant\rho_2(u,v)) \text{ for all $u,v\in V(G)$}.
\end{equation*}
\end{definition}

\begin{proposition}[\!\!\cite{DMV}]\label{p7}
Let $(G,w)$ be a nonempty weighted graph with a pseudometrizable weight
$w$. If $G$ is connected then the shortest-path
pseudometric $d_w$ belongs to $\mathfrak{M}_w$ and this pseudometric
is the greatest element of the poset $(\mathfrak{M}_w,\leqslant)$,
i.e., the inequality
\begin{equation}\label{e15}
\rho\leqslant d_w
\end{equation}
holds for each $\rho\in\mathfrak{M}_w$. Conversely, if the poset
$(\mathfrak{M}_w,\leqslant)$ contains the greatest element, then $G$
is connected.
\end{proposition}

\section{Uniqueness of continuation}

Clearly, if a weighted graph $(G,w)$ is not connected, then the continuation of the weight $w$ is not unique. Thus, it has sense to consider only connected graphs $G$. For the convergent sequence $(a_n)$ of reals we write $a_n \rightarrow a-0$ if $a_n \rightarrow a$ and $a_n\leqslant a$ for all $n\in \NN$. If, for a weighted graph $(G,w)$, the shortest-path pseudometric $d_w$ is a metric, then we say that $w$ is \emph{metrizable}.

\begin{theorem}\label{p1}
Let $(G,w)$ be a connected weighted graph with the metrizable weight $w$. 
The set $\mathfrak M_w$ consists only of one element $\rho$ if and only if for every two non-adjacent vertices $u, v\in V(G)$ there exists a sequence $(P_n)$ of paths $P_n\in \mathcal P_{u,v}$ (possibly stationary) such that
\begin{equation}\label{e2}
2\max\limits_{e\in P_n}w(e)-w(P_n)\rightarrow d_w(u,v)-0, \  \ \text{ as } \  n \rightarrow \infty.
\end{equation}
Moreover, if condition~(\ref{e2}) holds and $w$ is metrizable weight then $\rho(u,v)=d_w(u,v)$.
\end{theorem}

\begin{proof}
Let condition~(\ref{e2}) hold and let $w$ be a metrizable weight.
Definition of metrizability of $w$ means that $d_w(u,v)>0$ for all $u, v$, $u\neq v$.
Suppose there exists $\rho'\in \mathfrak M_w$ such that  $\rho'\neq d_w$. Then by Proposition~\ref{p7} the inequality
\begin{equation}\label{e211}
 \rho'(u,v)=b<d_w(u,v)
\end{equation}
holds for some non-adjacent $u,v \in V(G)$.

According to~(\ref{e2}) for every $\ve>0$ there exists $n_0\in \NN$ such that
\begin{equation}\label{e21}
2\max\limits_{e\in P_{n_0}}w(e)-w(P_{n_0}) >d_w(u,v) - \ve.
\end{equation}

 Let $\ve$ be  such that
\begin{equation}\label{e22}
d_w(u,v)>b+\ve.
\end{equation}
It follows from~(\ref{e21}) and~(\ref{e22}) that
\begin{equation*}
2\max\limits_{e\in P_{n_0}}w(e)-w(P_{n_0})>b.
\end{equation*}
Transform this inequality to the following
\begin{equation}\label{e24}
\max\limits_{e\in P_{n_0}}w(e)>w(P_{n_0})-\max\limits_{e\in P_{n_0}}w(e)+b.
\end{equation}
Consider the weighted  cycle $(C_{n_0}, w')$ such that $E(C_{n_0})=E(P_{n_0})\cup\{u,v\}$,  $w'(e)=w(e)$, $e\in P_{n_0}$ and $w'\{u,v\}=b$.
The value in the right side of inequality~(\ref{e24}) is the sum of all weights of edges of the cycle $C_{n_0}$ except the maximal weight.
Hence, $\max\limits_{e\in P_{n_0}}w(e)$ is the maximal weight of the cycle $C_{n_0}$.
Inequality~(\ref{e24}) implies
$$
2\max\limits_{e\in C{n_0}}w'(e)>w'(C_{n_0})
$$ and according to condition (iii) of Theorem~\ref{t1} this contradicts to the fact that $\rho'$ is a pseudometric.

Conversely, let the set $\mathfrak M_w$ consist only of one element. By Proposition~\ref{p7} this element is $d_w$. Suppose that for some non-adjacent $u,v \in V(G)$ and for all sequences $(P_n)$ of paths $P_n\in \mathcal P_{u,v}$ condition~(\ref{e2}) does not hold.
Let us show that there exists another continuation $\tilde{\rho} \in \mathfrak M_w$ of the weight $w$ such that $\tilde{\rho}(u,v)<d_w(u,v)$.

By supposition, there exists $\ve>0$ such that for all  paths $P \in \mathcal P_{u,v}$ the relation
$$
q(P)\in (-\infty, d_w(u,v)-\ve)\cup(d_w(u,v), +\infty)
$$
holds, where
\begin{equation}\label{e35}
q(P)=2\max\limits_{e\in P}w(e)-w(P).
\end{equation}

Suppose that for some $P$ the relation $q(P)\in (d_w(u,v), +\infty)$ holds. Consequently,~(\ref{e35}) implies
\begin{equation}\label{e36}
2\max\limits_{e\in P}w(e)>d_w(u,v)+w(P).
\end{equation}
Consider the weighted cycle $(C, w')$ defined as $E(C)=E(P)\cup\{u,v\}$, $w'(e)=w(e)$, $e\in E(P)$, $w'(\{u,v\})=d_w(u,v)$. Inequality~(\ref{e36}) implies
$$
2\max\limits_{e\in C}w'(e)>w'(C).
$$
Hence, according to condition (iii) of Theorem~\ref{t1} the function $d_w$ is not a metric, which is false by the supposition of the theorem.
Consequently, the relation $q(P)\in (d_w(u,v),+\infty)$ is impossible

Suppose now that for every $P$ we have $q(P)\in (-\infty,d_w(u,v)-\ve)$. Which means
\begin{equation*}\label{e41}
-\infty<2\max\limits_{e\in P}w(e)-w(P)<d_w(u,v)-\ve.
\end{equation*}
Hence, we have
\begin{equation}\label{e5}
2\max\limits_{e\in P}w(e)<w(P)+d_w(u,v)-\ve.
\end{equation}

Let $r$ be any real number from the interval $(d_w(u,v)-\ve, d_w(u,v))$ and let a weighted cycle $(C,\tilde{w})$ be such that $E(C)=E(P)\cup\{u,v\}$, $\tilde{w}(e)=w(e)$ for all $e\in E(P)$ and $\tilde{w}(\{u,v\})=r$.

Consider first the case where $\max\limits_{e\in P}w(e)>r$. Hence, $\max\limits_{e\in P}w(e)$ is a maximal weight in the cycle $C$.  Inequa\-li\-ty~(\ref{e5}) implies
$$
2\max\limits_{e\in C}\tilde{w}(e)= 2\max\limits_{e\in P}w(e)<w(P)+r=\tilde{w}(C).
$$
By condition (iii) of Theorem~\ref{t1} this means that the cycle $(C, \tilde{w})$ is metrizable.

Under the supposition  $\max\limits_{e\in P}w(e)\leqslant r$ we see that $\{u,v\}$ is an edge of maximal weight in the cycle $C$.
Hence,
$$
2\tilde{w}(\{u,v\})=2r<r+d_w(u,v)\leqslant r+w(P)=\tilde{w}(C).
$$
And the cycle $C$ is also metrizable.

Thus, the cycle $(C,\tilde{w})$ is metrizable for every $P\in \mathcal{P}_{u,v}$.

Consider a weighted graph $(G',w')$ such that $V(G')=V(G)$, $E(G')=E(G)$, $w'(e)=w(e)$ for every $e\in E(G)$ and $w'(\{u,v\})=r$. By condition (iii) of Theorem~\ref{t1} the weight $w'$ is metrizable since all the cycles in $G'$ are metrizable. The shortest-path pseudometric $d_{w'}$ considered on the graph $G'$ is one of the continuations of $w'$ which is a metric. It is clear that $d_{w'}$ is also a continuation of the weight $w$. Hence, there exist at least two different continuations $d_{w'}$ and $d_w$ of the weight $w$, since $r=d_{w'}(\{u,v\})<d_w(\{u,v\})$.
\end{proof}

Note that condition~(\ref{e2}) of Theorem~\ref{p1} can be reformulated as follows:
for every two non-adjacent vertices $u,v \in V(G)$ the relation
\begin{equation}\label{e10}
\sup\limits_{P\in \mathfrak P_{u,v}} \{2\max\limits_{e\in P}w(e)-w(P)\}= d_w(u,v)
\end{equation}
holds.
\begin{corollary}
Let $(G,w)$ be a connected weighted graph with a pseudometrizable weight $w$.
The weight $w$ admits the unique continuation to a pseudometric which is not a metric if and only if the two following conditions hold simultaneously:
 \begin{itemize}
   \item [(i)] There exist $u,v\in V(G)$, $u\neq v$, such that $d_w(u,v)=0$;
   \item [(ii)] For every non-adjacent $u,v\in V(G)$ such that $d_w(u,v)\neq 0$ condition~(\ref{e10}) holds.
 \end{itemize}
Moreover, if $\rho$ is the unique continuation of $w$, then $\rho(u,v)=d_w(u,v)$.
\end{corollary}

\begin{example}
\begin{figure}[ht]
\begin{tikzpicture}
\draw (0,0) node[left]{$u$}  -- (6,2) node[right]{$t$} -- (6,0) node[right]{$v$} -- (3,-1) node[below]{$s$} -- (0,0);

\draw[dashed] (0,0) -- (6,0);
\draw[dashed] (3,-1) -- (6,2);

\draw (3,1) node [above] {$5$};
\draw (6,1) node [right] {$1$};
\draw (4.5,-1) node [above] {$2$};
\draw (1,-0.7) node [right] {$2$};

\draw[fill=black] (0,0) circle [radius=0.09];
\draw[fill=black] (6,2) circle [radius=0.09];
\draw[fill=black] (3,-1) circle [radius=0.09];
\draw[fill=black] (6,0) circle [radius=0.09];
\end{tikzpicture}
\caption{An example of a finite weighted graph $(G,w)$ with the unique continuation.}
\label{f1}
\end{figure}
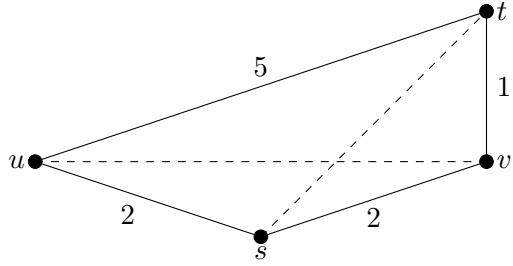

Let us consider a weighted graph $(G,w)$ such that $V(G)=\{u,v,s,t\}$, $w(\{u,t\})=5$, $w(\{t,v\})=1$, $w(\{v,s\})=w(\{s,u\})=2$. Consider infinite stationary sequences of paths $P_n=\{s,u,t\}$, $Q_n=\{u,t,v\}$ for every $n\in \NN$, see Figure~\ref{f1}. Since
$$
2\max\limits_{e\in P_n}w(e)-w(P_n)=2\cdot5-(5+2)=3= d_w(s,t)
$$
and
$$
2\max\limits_{e\in Q_n}w(e)-w(Q_n)=2\cdot5-(5+1)=4= d_w(u,v)
$$
for every $n$, by Proposition~\ref{p1} a continuation of the weight $w$ to a metric is unique.
\end{example}

\begin{example}
\begin{figure}[ht]
\begin{tikzpicture}
\draw (0,0) node[left]{$u$}  -- (3,5) node[above]{$x_2$} -- (6,0) node[right]{$v$} -- (3,-1) node[below]{$x_1$} -- (0,0);

\draw (0,0) -- (3,3) node[above]{$x_3$} -- (6,0);
\draw (0,0) -- (3,1) node[above]{$x_n$} -- (6,0);
\draw[loosely dotted] (3,3) -- (3,0);

\draw (1.3,2.5) node [above] {$5$};
\draw (4.9,2.5) node [above] {$1+\frac12$};

\draw (1.5,1.6) node [above] {$5$};
\draw (4.5,1.6) node [above] {${\scriptstyle 1+\frac13}$};

\draw (1.5,0.5) node [above] {$5$};
\draw (4.5,0.5) node [above] {$1+\frac1n$};

\draw (4.5,-1) node [above] {$2$};
\draw (1,-0.7) node [right] {$2$};

\draw[fill=black] (0,0) circle [radius=0.09];
\draw[fill=black] (3,5) circle [radius=0.09];
\draw[fill=black] (6,0) circle [radius=0.09];
\draw[fill=black] (3,3) circle [radius=0.09];
\draw[fill=black] (3,1) circle [radius=0.09];
\draw[fill=black] (3,-1) circle [radius=0.09];
\end{tikzpicture}
\caption{A weighted graph $(G,w)$ with $\card(V(G))=\aleph_0$.}
\label{f2}
\end{figure}
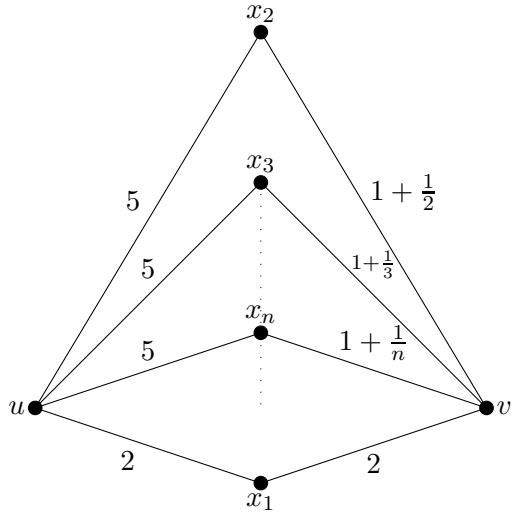

Let us consider an infinite case. Let $(G,w)$ be a weighted graph  with $V(G)=\{u,v,x_1,x_2,...\}$ and let $w(\{u,x_1\})=w(\{x_1,v\})=2$ and $w(\{u,x_n\})=5$, $w(\{v,x_n\})=1+1/n$, $n=2,3,...$, see Figure~\ref{f2}. By condition (iii) of Theorem~\ref{t1} the weight $w$ is pseudometrizable. Moreover, it is easy to see that it is also metrizable.

Consider a weighted graph $(\widetilde{G},\widetilde{w})$ such that $V(\widetilde{G})=V(G)$, and the set of edges $E(\widetilde{G})$ consists of all distinct unordered pairs of points from $V(G)$ except the pair $\{u,v\}$. For every edge $\{x,y\}\in E(\widetilde{G})$ define its weight $\widetilde{w}(\{x,y\})=d_w(u,v)$ where $d_w$ is a shortest-path pseudometric of the graph $(G,w)$, see~(\ref{e3}).

Consider a sequence of paths $P_n=\{u,x_n,v\}$.
Since the relation
$$
2\max\limits_{e\in P_n}\widetilde{w}(e)-\widetilde{w}(P_n)
=2\cdot5-(5+1+\frac{1}{n})
\rightarrow
4-0=
d_{\widetilde{w}}(u,v)-0, \ \text{ as } \  n \rightarrow \infty,
$$
holds, we have that the weight $\tilde{w}$ has a unique continuation to a metric $\rho$ which must be established only for the edge $\{u,v\}$, i.e., $\rho(u,v)=d_{\tilde{w}}(u,v)=4$.
\end{example}


\begin{thebibliography}{1}

\bibitem{DMV}
O.~Dovgoshey, O.~Martio, and M.~Vuorinen.
\newblock Metrization of weighted graphs.
\newblock {\em Ann. Comb.}, 17(3):455--476, 2013.

\bibitem{DP}
A.~A. Dovgoshe\u{\i} and E.~A. Petrov.
\newblock A subdominant pseudo-ultrametric on graphs.
\newblock {\em Mat. Sb.}, 204(8):51--72, 2013.

\bibitem{B}
S.~A. Bogaty\u{\i}.
\newblock Metrically homogeneous spaces.
\newblock {\em Uspekhi Mat. Nauk}, 57(2(344)):3--22, 2002.

\bibitem{AMI}
A.~Ivanov and B.~Majcher-Iwanow.
\newblock An amalgamation property for metric spaces.
\newblock {\em Algebra Discrete Math.}, 22(2):233--239, 2016.

\bibitem{H}
M.~Hu\v{s}ek.
\newblock Extension of mappings and pseudometrics.
\newblock {\em Extracta Math.}, 25(3):277--308, 2010.

\bibitem{BM}
J.~A. Bondy and U.~S.~R. Murty.
\newblock {\em Graph theory}, volume 244 of {\em Graduate Texts in
  Mathematics}.
\newblock Springer, New York, 2008.

\end{thebibliography}
\end{document}